\documentclass{amsart}
\usepackage{amssymb}
\usepackage{color}
\usepackage[allcolors = blue, colorlinks = true]{hyperref}
\usepackage{mathptmx}
\usepackage{amsrefs}
\usepackage{comment}
\usepackage{mathtools}
\usepackage{esint}
\newtheorem{theorem}{Theorem}[section]
\newtheorem{lemma}[theorem]{Lemma}
\newtheorem{corollary}[theorem]{Corollary}

\newtheorem{proposition}[theorem]{Proposition}
\theoremstyle{definition}
\newtheorem{definition}{Definition}
\theoremstyle{remark}
\newtheorem{remark}{Remark}[section]
\numberwithin{equation}{section}
\newcommand{\R}{\mathbb{R}}

\newcommand{\Om}{\Omega}
\newcommand{\Rn}{\R^n}
\newcommand{\il}{\Delta_{\infty}}

\newcommand{\la}{\langle}
\newcommand{\ra}{\rangle}
\newcommand{\loc}{\textnormal{loc}}
\DeclareMathOperator{\spt}{spt\,}
\DeclareMathOperator{\tr}{tr\,}
\DeclareMathOperator{\diverg}{div\,}
\begin{document}
\title[Second order estimates]{Global second order Sobolev-regularity of $p$-harmonic functions}
\author{Akseli Haarala}
\address{Matematiikan ja tilastotieteen laitos, Helsingin yliopisto, PL 68 (Pietari Kalmin katu 5) 00014 Helsingin yliopisto, Helsinki, Finland}
\email[Akseli Haarala]{akseli.haarala@helsinki.fi}
\author{Saara Sarsa}
\address{Matematiikan ja tilastotieteen laitos, Helsingin yliopisto, PL 68 (Pietari Kalmin katu 5) 00014 Helsingin yliopisto, Helsinki, Finland}
\email[Saara Sarsa]{saara.sarsa@helsinki.fi}
\thanks{The authors were both supported by the Academy of Finland, project 308759. Saara Sarsa is supported by the Academy of Finland, Center of Excellence in Randomness and Structures.}
\begin{abstract}
We prove a global version of the classical result that $p$-harmonic functions belong to $W^{2,2}_{\loc}$ for $1<p<3+\frac{2}{n-2}$. The proof relies on Cordes' matrix inequalities \cite{Cordes1961} and techniques from the work of Cianchi and Maz'ya \cites{Cianchi2018,Cianchi2019}.
\end{abstract}
\maketitle
\section{Introduction}

Let $\Om\subset\Rn$, $n\geq 2$, be a bounded Lipschitz domain and $1<p<\infty$. 
For some $\varphi\in W^{1,p}(\Om)$ consider the problem of minimizing the $p$-energy functional 
$$ \int_\Om|Dv|^pdx $$
among all functions $v\in \varphi + W^{1,p}_0(\Om):=\{v\in W^{1,p}(\Om):v-\varphi\in W^{1,p}_0(\Om)\}$. Here $Dv=(v_{x_1},\ldots,v_{x_n})$ denotes the weak gradient of $v$ and $|Dv|=(v_{x_1}^2+\ldots+v_{x_n}^2)^{1/2}$ its Euclidean norm.
By the direct method of calculus of variations, there exists a unique minimizer $u\in \varphi+W^{1,p}_0(\Om)$. Moreover, the minimizer $u$ is $p$-harmonic in $\Om$. That is, $u$ solves the homogeneous $p$-Laplacian equation
\begin{equation} \label{eq:p-Laplace}
    \Delta_pu :=\diverg(|Du|^{p-2}Du)=0
\end{equation}
in the weak sense. More precisely, 
\begin{align*}
    \int_\Om |Du|^{p-2}\la Du,D\phi\ra dx=0
\end{align*}
for all $\phi\in C^\infty_0(\Om)$.
Thus the minimizer $u$ is a weak solution to the Dirichlet problem
\begin{equation} \label{eq:Dirichlet-problem-homogeneous-Intro}
\begin{cases}
\begin{aligned}
\Delta_pu=0 &\quad\text{in }\Om; \\
u=\varphi &\quad\text{on }\partial \Om.
\end{aligned}
\end{cases}
\end{equation}
For the definitions of the function spaces that are ubiquitous in this paper, see the end of this section.

The regularity theory of $p$-harmonic functions is well known.
Any $p$-harmonic function belongs to $C^{1,\alpha}_{\loc}$, where $\alpha=\alpha(n,p)\in (0,1]$, see \cites{Uraltseva1968,Uhlenbeck1977,Evans1982} for the case $p\geq2$ and \cites{Lewis1983,DiBenedetto1983,Tolksdorf1984,Iwaniec1989,Wang1994} for the full range $1<p<\infty$. 
The standard second order regularity result of $p$-harmonic functions states that $|Du|^{\frac{p-2}{2}}Du\in W^{1,2}_{\loc}$, see
\cite{Bojarski1987} for $p\geq2$ and \cite{DiBenedetto1989} for $1<p<\infty$.
Manfredi and Weitsman \cite{Manfredi1988}*{Lemma 5.1} showed that $u\in W^{2,2}_{\loc}(\Om)$ under the condition $1<p<3+\frac{2}{n-2}$. The restriction to the range of $p$ arises from a so-called Cordes condition.
We also mention a recent result by Dong, Peng, Zhang and Zhou \cite{Dong2020} that $|Du|^\beta Du\in W^{1,2}_{\loc}$ for $\beta>-1+\frac{(n-2)\max\{p-1,1\}}{2(n-1)}$. 
The range of $\beta$ was later improved to $\beta>-1+\frac{(n-2)(p-1)}{2(n-1)}$ by the second author \cite{Sarsa2022}.

The purpose of this paper is to investigate a global version of the classical result by Manfredi and Weitsman \cite{Manfredi1988}*{Lemma 5.1} that if $1<p<3+\frac{2}{n-2}$, then $u\in W^{2,2}_{\loc}(\Om)$. 
We prove that if $1<p<3+\frac{2}{n-2}$ and $\varphi\in W^{1,p}(\Om)\cap W^{2,2}(\Om)$ and the boundary $\partial\Om$ satisfies certain additional regularity assumptions, then $u\in W^{2,2}(\Om)$. 

For the regularity assumptions on the boundary $\partial\Om$, we refer to two recent papers by Cianchi and Maz'ya \cites{Cianchi2018,Cianchi2019}. See also \cite{Balci2021}. 

In \cite{Cianchi2018} the authors study the global second order Sobolev-regularity of the so-called generalized solutions to
\begin{equation} \label{eq:CM-PDE}
    -\diverg(a(|Du|)Du)=f\quad\text{in }\Om 
\end{equation}
with both Dirichlet boundary condition $u=0$ on $\partial\Om$ and Neumann boundary condition $\frac{\partial u}{\partial\nu}=0$ on $\partial\Om$.
Here $a\colon(0,\infty)\to(0,\infty)$ is a $C^1$-function such that
\begin{equation} \label{eq:CM-ellipticity-bound-Intro}
    -1<i_a:=\inf_{t>0}\,\frac{ta'(t)}{a(t)}\leq 
    \sup_{t>0}\,\frac{ta'(t)}{a(t)}=:s_a<\infty
\end{equation}
and $a'$ denotes the derivative of $a$.
The authors prove that $a(|Du|)Du\in W^{1,2}(\Om)$ if and only if $f\in L^2(\Om)$, provided that the boundary $\partial\Om$ of the Lipschitz domain $\Om$ is sufficiently regular. 
For the regularity of the boundary $\partial\Om$, firstly they assume that $\partial\Om\in W^{2,1}$, that is, the boundary $\partial\Om$ can be locally viewed as a graph of a twice weakly differentiable function. Secondly, they assume that the weak second fundamental form of $\partial\Om$ satisfies certain summability assumptions. 
In \cite{Cianchi2019} the authors prove a similar result in the case when the equation \eqref{eq:CM-PDE} is replaced by the $p$-Laplace system. They also relax the boundary regularity assumption compared to \cite{Cianchi2018}. 

Our regularity assumption on the boundary $\partial\Om$ is the same one that was used in \cite{Cianchi2019}. 
We leave the more detailed discussion concerning the boundary regularity results from \cites{Cianchi2018,Cianchi2019} to Section \ref{sec:BoundaryEstimates}. Here we merely introduce the notation that is needed to state our main theorem. Our notation is adopted from \cites{Cianchi2019}. 

We denote the diameter of $\Om$ by $d_\Om$ and the Lipschitz constant of the boundary $\partial\Om$ by $L_\Om$.
Suppose that $\partial\Om\in W^{2,1}$.
In particular, the second fundamental form $\mathcal{B}\colon\partial\Om\to\R^{(n-1)\times(n-1)}$ exists in the weak sense. Denote its norm by $|\mathcal{B}|$.

Let us define $\mathcal{K}_{\Om}\colon(0,1)\to[0,\infty]$ as
\begin{align} \label{eq:Kquantity-Intro}
    \mathcal{K}_{\Om}(r):=\sup_{x\in\partial\Om}
    \sup_{E\subset\partial\Om\cap B_r(x)}
    \frac{\int_E|\mathcal{B}|d\mathcal{H}^{n-1}}{\operatorname{cap}_{B_1(x)}(E)},
\end{align}
where
\begin{align*}
    \operatorname{cap}_{B_1(x)}(E)
    :=\inf\Big\{
    \int_{\Rn}|Dv|^2dx:v\in C^1_0(B_1(x))\text{ and } v\geq 1\text{ in }E
    \Big\}
\end{align*}
denotes the capacity of a set $E$ relative to the ball $B_1(x)$ and $\mathcal{H}^{n-1}$ stands for the $(n-1)$-dimensional Hausdorff measure.

We can now state our main theorem.

\begin{theorem} \label{thm:Global-estimate}
Let $\Om\subset\Rn$ be a bounded Lipschitz domain with diameter $d_\Om>0$ and Lipschitz constant $L_\Om>0$. 
Suppose in addition that $\partial\Om\in W^{2,1}$. 
Assume that $1<p<3+\frac{2}{n-2}$ and let $u\in W^{1,p}(\Om)$ solve \eqref{eq:Dirichlet-problem-homogeneous-Intro} with $\varphi\in W^{1,p}(\Om)\cap W^{2,2}(\Om)$. 
There exists a constant $\mathcal{K}_0=\mathcal{K}_0(n,p,d_\Om,L_\Om)>0$ such that if
\begin{equation} \label{eq:Smallness-of-the-limit-of-Kquantity}
    \lim_{r\to 0}\mathcal{K}_{\Om}(r)<\mathcal{K}_0,
\end{equation} 
then $Du\in W^{1,2}(\Om;\Rn)$ and
\begin{equation} \label{eq:Global-estimate}
\begin{aligned}
    \|Du\|_{W^{1,2}(\Om;\Rn)}
    \leq C\Big(\|D\varphi\|_{W^{1,2}(\Om;\Rn)}+\|D\varphi\|_{L^p(\Om;\Rn)}\Big)
\end{aligned}
\end{equation}
for some $C=C(n,p,\Om)$.
\end{theorem}

The proof of Theorem \ref{thm:Global-estimate} is based on Cordes' matrix inequalities \cite{Cordes1961} and the aforementioned techniques taken from the work of Cianchi and Maz'ya \cites{Cianchi2018,Cianchi2019}.

To illustrate the idea of the proof of Theorem \ref{thm:Global-estimate}, consider the well-known identity
\begin{equation} \label{eq:Basic-identity-for-Laplacian}
    |D^2u|^2=\diverg(D^2uDu-\Delta uDu)
    +(\Delta u)^2
\end{equation}
that holds for any function $u\in C^3(\Om)$. 
Here $D^2u=(u_{x_ix_j})_{i,j=1}^n$ denotes the Hessian matrix of $u$, $|D^2u|=(\sum_{i,j=1}^nu_{x_ix_j}^2)^{1/2}$ denotes its (Hilbert-Schmidt) norm, and $\Delta u=\sum_{i=1}^nu_{x_ix_i}$ denotes the Laplacian of $u$.
The identity \eqref{eq:Basic-identity-for-Laplacian} appears repeatedly in the literature, see for instance \cites{Talenti1965,Grisvard1985,Koch2019,Dong2020,Balci2021}. 

The identity \eqref{eq:Basic-identity-for-Laplacian} has been generalized by replacing the Laplacian $\Delta u$ by a more general elliptic operator.

Let $A\colon\Om\to\R^{n\times n}$ be a matrix-valued function, $A=(A_{ij})_{i,j=1}^n$. Denote the norm of $A$ by $|A|=(\sum_{i,j=1}^nA_{ij}^2)^{1/2}$ and the trace of $A$ by $\tr(A)=\sum_{i=1}^nA_{ii}$.
Talenti \cite{Talenti1965} showed that if $A$ satisfies the so-called Cordes condition 
\begin{equation} \label{eq:Cordes-condition}
    (n-1+\delta)|A|^2\leq(\tr A)^2 \quad\text{in }\Om
\end{equation}
for some $0<\delta<1$, then for any $u\in C^3(\Om)$
\begin{equation}\label{eq:Talenti}
    c|D^2u|^2\leq \diverg(D^2u Du-\Delta u Du)
    +\frac{C}{|A|^2}\la A,D^2u\ra^2
\end{equation}
where $c=c(n,\delta)>0$ and $C=C(n,\delta)>0$. 
In the display \eqref{eq:Talenti} the brackets $\la\cdot,\cdot\ra$ denote the inner product of two matrices, that is, $\la A,D^2u\ra=\sum_{i,j=1}^nA_{ij}u_{x_ix_j}$. If $A$ happens to be the identity matrix, one can see that indeed, \eqref{eq:Talenti} can be viewed as a generalization of \eqref{eq:Basic-identity-for-Laplacian}.

Cianchi and Maz'ya \cite{Cianchi2018}*{Lemma 3.1} proved another generalization of \eqref{eq:Basic-identity-for-Laplacian}.
If $a\colon[0,\infty)\to(0,\infty)$ is a (sufficiently well-behaving) $C^1$-function such that the first inequality in \eqref{eq:CM-ellipticity-bound-Intro} holds,
then for any $u\in C^3(\Om)$
\begin{equation}\label{eq:CM-pointwise-differential-inequality}
    c\big(a(|Du|)\big)^2|D^2u|^2
    \leq
    \diverg\big(a(|Du|)^2(D^2u Du-\Delta u Du)\big) 
    +
    \big(\diverg(a(|Du|)Du)\big)^2
\end{equation}
where $c=c(n,i_a)>0$. For the precise assumptions on the function $a$, we refer to the statement of Lemma 3.1 in \cite{Cianchi2018}. Similarly as \eqref{eq:Talenti}, the inequality \eqref{eq:CM-pointwise-differential-inequality} can be viewed as a generalization of \eqref{eq:Basic-identity-for-Laplacian}.

Our central observation is that Cordes' matrix inequalities \cite{Cordes1961} imply yet another generalization of \eqref{eq:Basic-identity-for-Laplacian}, that contains both \eqref{eq:Talenti} and \eqref{eq:CM-pointwise-differential-inequality} as a special case.

Namely, let $a,b\colon[0,\infty)\to(0,\infty)$ be $C^1$-functions and denote
$$ \vartheta_a(t):=\frac{ta'(t)}{a(t)} 
\quad\text{and}\quad
\vartheta_b(t):=\frac{tb'(t)}{b(t)}
\quad\text{for all }t\geq0. $$
For any $u\in C^3(\Om)$, consider the vector fields
$$ V_a:=a(|Du|)Du\quad\text{and}\quad V_b:=b(|Du|)Du. $$
We show that if $a$ and $b$ satisfy the growth conditions
\begin{align*}
    -1<\inf_{t\geq0}\vartheta_a(t)\leq \sup_{t\geq 0}\vartheta_a(t)<\infty
    \quad\text{and}\quad
    -1<\inf_{t\geq0}\vartheta_b(t)\leq \sup_{t\geq 0}\vartheta_b(t)<\infty,
\end{align*}
and if they are "sufficiently close to each other", that is, if
\begin{align} \label{eq:Intro-Main-condition}
    0<\inf_{t\geq 0}\,
    \frac{1+\vartheta_a(t)}{1+\vartheta_b(t)}
    \leq \sup_{t\geq0}\,
    \frac{1+\vartheta_a(t)}{1+\vartheta_b(t)}
    <\frac{2(n-1)}{n-2},
\end{align}
then
\begin{align} \label{eq:Intro-Inequality}
    c|DV_b|^2
    &\leq
    \diverg\big((DV_b-\tr(DV_b)I)V_b\big) 
    +C\Big(\frac{b(|Du|)}{a(|Du|)}\Big)^2\big(\diverg(V_a)\big)^2,
\end{align}
for some positive constants $c>0$ and $C>0$. 
Our proof of \eqref{eq:Intro-Inequality} is based on Cordes' matrix inequalities which we discuss in Section \ref{sec:Cordes}. The proof of inequality \eqref{eq:Intro-Inequality} is given in Section \ref{sec:Differential-identity}, see Lemma \ref{lem:Inequality}.

The inequality \eqref{eq:Intro-Inequality} is the key tool of this paper.
It is a generalization of the aforementioned inequalities \eqref{eq:Talenti} and \eqref{eq:CM-pointwise-differential-inequality}.
To recover Talenti's inequality \eqref{eq:Talenti}, put $b\equiv 1$ in \eqref{eq:Intro-Inequality}. 
On the other hand, to recover Cianchi and Mazy'a's inequality \eqref{eq:CM-pointwise-differential-inequality}, put $b=a$ in \eqref{eq:Intro-Inequality}. 

In fact, our proof of \eqref{eq:Intro-Inequality} with $b=a$ is simpler and shorter than that of \eqref{eq:CM-pointwise-differential-inequality} in \cite{Cianchi2018}. See Remark \ref{rem:Recovery-of-CM} for details. In Remark \ref{rem:Recovery-of-CM} we also explain why in the case $b=a$ we can choose $C=1$ in \eqref{eq:Intro-Inequality}.

To prove the estimate \eqref{eq:Global-estimate} in Theorem \ref{thm:Global-estimate}, we use a further generalization of \eqref{eq:Intro-Inequality}. 
Namely, if \eqref{eq:Intro-Main-condition} holds and $W\in C^2(\Om;\Rn)$ is an arbitrary vector field, then
\begin{equation} \label{eq:Intro-Inequality-2} 
\begin{aligned}
    c|DV_b|^2
    &\leq
    \diverg\big((D(V_b-W)-\tr(D(V_b-W))I)(V_b-W)\big) \\
    &\quad
    +C\Big(|DW|^2+\Big(\frac{b(|Du|)}{a(|Du|)}\Big)^2\big(\diverg(V_a)\big)^2\Big).
\end{aligned}    
\end{equation}
See Corollary \ref{cor:Inequality-with-arbitrary-vectorfield} for details.

Observe that the inequality \eqref{eq:Intro-Inequality-2} can be used to easily derive local $L^{2}$-estimates for $DV_b$ in terms of $L^2$-oscillation of $V_b$, see Remark \ref{rem:Previous-work}. 

Let us briefly explain the outline of the proof of Theorem \ref{thm:Global-estimate}.
Once the inequality \eqref{eq:Intro-Inequality-2} is established in Sections \ref{sec:Cordes} and \ref{sec:Differential-identity}, we follow the idea of the proof of Theorem 2.4 from \cite{Cianchi2018}.
We first derive the estimate \eqref{eq:Global-estimate} under additional regularity assumptions on the domain $\Om$, the boundary data $\varphi$ and the operator $\Delta_p$. For $\epsilon>0$ small, we set $a(t)=(t^2+\epsilon)^{\frac{p-2}{2}}$ and $b\equiv1$ in \eqref{eq:Intro-Inequality-2}. In this case the condition \eqref{eq:Intro-Main-condition} is satisfied precisely when
$$ 1<p<3+\frac{2}{n-2}. $$ 

The key idea is to set $W=D\varphi$ in \eqref{eq:Intro-Inequality-2}. 
Since $u=\varphi$ on $\partial\Om$, we have $D_Tu=D_T\varphi$ on $\partial\Om$.
Here $D_T$ refers to the tangential gradient on the boundary, see Section \ref{sec:Boundary-Identity} for details.
In other words, the vector field $Du-D\varphi$ is always normal to the boundary manifold $\partial\Om$.
The flow of such vector field across the boundary $\partial\Om$ can be estimated in terms of the second fundamental form of the boundary manifold.
Consequently, we may derive appropriate boundary estimates similarly as in \cites{Cianchi2018,Cianchi2019}. 

In Section \ref{sec:BoundaryEstimates} we state a weighted trace inequality from \cite{Cianchi2019} that explains how the quantity $\mathcal{K}_\Om$ comes into play. In Section \ref{sec:BoundaryEstimates} we also make some remarks about the meaning of the condition \eqref{eq:Smallness-of-the-limit-of-Kquantity}.

In Section \ref{sec:Regularization} we put the tools from the previous sections into use to derive a regularized version of the estimate \eqref{eq:Global-estimate}, see Proposition \ref{prop:Global-estimate}.
Once the regularized version of \eqref{eq:Global-estimate} is established, we carefully pass to the limit to conclude the proof of Theorem \ref{thm:Global-estimate}. See Section \ref{sec:Proof} for details.

Throughout this paper, $W^{1,p}(\Om)$ denotes the Sobolev space of weakly differentiable functions $v\colon \Om\to\R$ such that $v,v_{x_i}\in L^p(\Om)$ for all $i=1,\ldots,n$. $W^{2,2}(\Om)$ denotes the Sobolev space of twice weakly differentiable functions $v\colon \Om\to\R$ such that $v,v_{x_i},v_{x_i,x_j}\in L^2(\Om)$ for all $i,j=1,\ldots,n$. $C^\infty_0(\Om)$ denotes the space of smooth functions that are compactly supported on $\Om$. $W^{1,p}_0(\Om)$ denotes the closure of $C^\infty_0(\Om)$ with respect to the Sobolev norm
$$ \|v\|_{W^{1,p}(\Om)}
=\Big(\int_\Om|v|^pdx+\sum_{i=1}^n\int_\Om|v_{x_i}|^pdx\Big)^{1/p}. $$

\section{Cordes' matrix inequalities} \label{sec:Cordes}

All the matrix inequalities in this section are due to Cordes \cite{Cordes1961}. See also \cite{Talenti1965}*{Lemma 3}. We provide detailed proofs for the convenience of the reader.

Consider the space of real square matrices $\R^{n\times n}$.
The space $\R^{n\times n}$ equipped with the usual Hilbert-Schmidt inner product 
$$ \la M,N\ra :=\sum_{i,j=1}^nM_{ij}N_{ij},
\quad\text{for all }M=(M_{ij})_{i,j=1}^n\in\R^{n\times n}\text{ and }
N=(N_{ij})_{i,j=1}^n\in\R^{n\times n}, $$
is a Hilbert space.
Given $M=(M_{ij})_{i,j=1}^n\in \R^{n\times n}$, we denote its Hilbert-Schmidt norm by $|M|:=\la M,M\ra^{1/2}$,
transpose by $M^\intercal:=(M_{ji})_{i,j=1}^n$ 
and trace by $\tr(M):=\sum_{i=1}^nM_{ii}$. If $M$ is invertible, we denote its inverse by $M^{-1}$. The identity matrix is denoted by $I:=\operatorname{diag}(1,\ldots,1)$.

\begin{definition}
A symmetric matrix $A\in\R^{n\times n}$ satisfies Cordes condition with respect to a symmetric, positive definite matrix $B\in\R^{n\times n}$ if $A\neq 0$ and
\begin{equation} \label{eq:Cordes-condition-wrtB}
    (n-1+\delta)\la B^{-1}A,(B^{-1}A)^\intercal\ra
    \leq (\tr(B^{-1}A))^2
\end{equation}
for some $0<\delta<1$.
\end{definition}

\begin{remark}
Any symmetric, positive definite matrix satisfies Cordes condition with respect to itself.
\end{remark}

\begin{lemma} \label{lem:Basic-Cordes}
Suppose that $A\in\R^{n\times n}$ is a symmetric matrix that satisfies Cordes condition \eqref{eq:Cordes-condition-wrtB} with respect to the identity matrix $I\in \R^{n\times n}$
for some $0<\delta<1$.
Then there exist constants $c=c(n,\delta)>0$ and $C=C(n,\delta)>0$ such that
\begin{align} \label{eq:Basic-form-of-Cordes}
    c|M|^2\leq |M|^2-(\tr(M))^2+\frac{C}{|A|^2}\la A,M\ra^2
\end{align}
for any symmetric matrix $M\in \R^{n\times n}$.
\end{lemma}

\begin{proof}
Write
$$ A=A_{\parallel}+A_{\perp} $$
where 
$$ A_{\parallel}:=\frac{\tr(A)}{n}I $$
is parallel to the identity matrix and 
$$ A_{\perp}:=A-A_{\parallel} $$ 
is perpendicular to the identity matrix.
Then the Cordes condition \eqref{eq:Cordes-condition-wrtB} with respect to the identity matrix can be written as
\begin{align} \label{eq:Another-form-of-Cordes}
    \delta|A|^2\leq |A_\parallel|^2-(n-1)|A_\perp|^2
\end{align}
for some $0<\delta<1$.
Notice that it suffices to study the case $A_\parallel\neq 0$ and $A_\perp\neq 0$. Namely, if $A_\parallel=0$, then \eqref{eq:Another-form-of-Cordes} cannot hold. On the other hand, if $A_\perp=0$, then the claim is trivially true for $c=1$ and $C=n$.

Let $M\in\R^{n\times n}$ be any symmetric matrix. We write 
\begin{align*}
    M=m_\parallel\frac{A_\parallel}{|A_\parallel|}+m_\perp\frac{A_\perp}{|A_\perp|}+\tilde{M}
\end{align*}
where
$$ m_\parallel:=\la \frac{A_\parallel}{|A_\parallel|},M\ra
\quad\text{and}\quad
m_\perp:=\la \frac{A_\perp}{|A_\perp|},M\ra $$
are the coordinates of $M$ with respect to $\operatorname{span}\{A_\parallel\}$ and $\operatorname{span}\{A_\perp\}$ respectively, and
$\tilde{M}$ denotes the part of $M$ that is perpendicular to both $I$ and $A$. Then
\begin{align*}
    |M|^2=m_\parallel^2+m_\perp^2+|\tilde{M}|^2 
\end{align*}
and
\begin{align*}
    (\tr(M))^2=nm_\parallel^2
\end{align*}
and
\begin{align*}
    \la A,M\ra^2
    &=\big(\la A_\parallel,M\ra+\la A_\perp,M\ra\big)^2 \\
    &=\big(|A_\parallel|m_\parallel+|A_\perp|m_\perp\big)^2 \\
    &=|A_\parallel|^2m_\parallel^2+2|A_\parallel||A_\perp|m_\parallel m_\perp+|A_\perp|^2m_\perp^2.
\end{align*}
Now the right hand side of \eqref{eq:Basic-form-of-Cordes} can be written as
\begin{align*}
    &|M|^2-(\tr(M))^2+\frac{C}{|A|^2}\la A,M\ra^2 \\
    &=
    \big(m_\parallel^2+m_\perp^2+|\tilde{M}|^2\big)
    -nm_\parallel^2
    +\frac{C}{|A|^2}\big(|A_\parallel|^2m_\parallel^2+2|A_\parallel||A_\perp|m_\parallel m_\perp+|A_\perp|^2m_\perp^2\big) \\
    &=Q+|\tilde{M}|^2
\end{align*}
where
\begin{align*}
    Q:=\Big(1-n+\frac{C|A_\parallel|^2}{|A|^2}\Big)m_\parallel^2
    +\Big(1+\frac{C|A_\perp|^2}{|A|^2}\Big)m_\perp^2 
    +\frac{2C|A_\parallel||A_\perp|}{|A|^2}m_\parallel m_\perp 
\end{align*}
is a quadratic form in $m _\parallel$ and $m_\perp$.
By the Cordes condition \eqref{eq:Another-form-of-Cordes}, the determinant of $Q$ has a lower bound 
\begin{align*}
    \operatorname{det}_Q
    &=
    1-n+\frac{C}{|A|^2}\big(|A_\parallel|^2-(n-1)|A_\perp|^2\big) \\
    &\geq
    1-n+C\delta. 
\end{align*}
Consequently, we can select $C=C(n,\delta)>0$ such that $\operatorname{det}_Q>0$. 
We can then fix $c=c(n,\delta)>0$ such that
$$ Q\geq c(m_{\parallel}^2+m_\perp^2) $$
and thus \eqref{eq:Basic-form-of-Cordes} holds.
\end{proof}

\begin{corollary} \label{cor:General-Cordes}
Let $A\in\R^{n\times n}$ be a symmetric matrix that satisfies the Cordes condition  \eqref{eq:Cordes-condition-wrtB} with respect to a symmetric, positive definite matrix $B\in\R^{n\times n}$
for some $0<\delta<1$.
Then there exist constants $c=c(n,\delta)>0$ and $C=C(n,\delta)>0$ such that
\begin{equation*}
    c\la BM,(BM)^\intercal\ra
    \leq 
    \la BM,(BM)^\intercal\ra-(\tr(BM))^2
    +\frac{C}{\la B^{-1}A,(B^{-1}A)^\intercal\ra}\langle A,M\rangle^2
\end{equation*}
for any symmetric matrix $M\in \R^{n\times n}$.
\end{corollary}

\begin{proof}
Since $B$ is symmetric and positive definite, there exists an invertible matrix $\Phi\in\R^{n\times n}$ such that
$\Phi^\intercal B\Phi=I$.
Consequently $\Phi^\intercal A\Phi$ is a symmetric matrix that satisfies Cordes condition \eqref{eq:Cordes-condition-wrtB} with respect to the identity matrix. 
We apply Lemma \ref{lem:Basic-Cordes} with $\Phi^\intercal A\Phi$ in place of $A$ and $\Phi^{-1}M(\Phi^{-1})^\intercal$ in place of $M$ to obtain the desired result. 
\end{proof}

\begin{lemma} \label{lem:Lower-bound-for-product-of-matrix-and-transpose}
Let $B\in\R^{n\times n}$ be a positive definite symmetric matrix with the smallest eigenvalue $\lambda_{\min}$ and the largest eigenvalue $\lambda_{\max}$. Then
$$ |BM|^2\leq \Big(\frac{\lambda_{\max}}{\lambda_{\min}}\Big)^2\la BM,(BM)^\intercal\ra $$
for any symmetric matrix $M\in\R^{n\times n}$.
\end{lemma}

\begin{proof}
Without loss of generality we may assume that $B$ is a diagonal matrix, that is, $B=\operatorname{diag}(\lambda_1,\ldots,\lambda_n)$ for some $\lambda_1\ldots,\lambda_n\in(0,\infty)$.
Then
\begin{align} \label{eq:Upper-bound}
    |BM|^2
    =\sum_{i,j}(\lambda_iM_{ij})^2
    \leq \lambda_{\max}^2|M|^2
\end{align}
and
\begin{align} \label{eq:Lower-bound}
    \la BM,(BM)^\intercal\ra
    =\sum_{i,j=1}^n\lambda_i M_{ij}\lambda_jM_{ji} 
    =\sum_{i,j=1}^n\lambda_i\lambda_jM_{ij}^2 
    \geq \lambda_{\min}^2|M|^2.
\end{align}
The estimates \eqref{eq:Upper-bound} and \eqref{eq:Lower-bound} together imply the desired result.
\end{proof}

\section{Differential inequality} \label{sec:Differential-identity}

In this section we prove our key inequalities \eqref{eq:Intro-Inequality} and \eqref{eq:Intro-Inequality-2}.
Let $\Om\subset\Rn$ be a domain and $u\in C^3(\Om)$. Consider the two vector fields
$$ V_a:=a(|Du|)Du\quad\text{and}\quad V_b:=b(|Du|)Du, $$
where $a\colon[0,\infty)\to(0,\infty)$ and $b\colon[0,\infty)\to(0,\infty)$ are $C^1$-functions.
We introduce some notation, similar to \cite{Cianchi2018}. Let
$$ \vartheta_a(t):=\frac{ta'(t)}{a(t)} 
\quad\text{and}\quad
\vartheta_b(t):=\frac{tb'(t)}{b(t)}
\quad\text{for all }t\geq0, $$
and
$$ i_a:=\inf_{t\geq0} \vartheta_a(t)\quad\text{and}\quad
s_a:=\sup_{t\geq 0}\vartheta_a(t), $$
and
$$ i_b:=\inf_{t\geq0} \vartheta_b(t)\quad\text{and}\quad
s_b:=\sup_{t\geq 0}\vartheta_b(t). $$
Our standing assumption on $a$ and $b$ is that
\begin{equation} \label{eq:Ellipticity-bounds}
    -1<i_a\leq s_a<\infty
    \quad\text{and}\quad
    -1<i_b\leq s_b<\infty.    
\end{equation}
Observe that
\begin{align*}
    DV_a=\big((a(|Du|)u_{x_i})_{x_j}\big)_{i,j=1}^n
    =a(|Du|)AD^2u
\end{align*}
and
\begin{align*}
    DV_b=\big((b(|Du|)u_{x_i})_{x_j}\big)_{i,j=1}^n
    =b(|Du|)BD^2u
\end{align*}
where
\begin{equation} \label{eq:AB}
    A:=I+\vartheta_a(|Du|)\frac{Du\otimes Du}{|Du|^2}
    \quad\text{and}\quad
    B:=I+\vartheta_b(|Du|)\frac{Du\otimes Du}{|Du|^2}
\end{equation}
are symmetric $n\times n$ matrices. In the above display $I$ stands for the identity matrix and $Du\otimes Du$ stands for the tensor product of the vector $Du$ with itself. 

We record certain elementary facts on the matrices $A$ and $B$. If $Du\neq 0$, the eigenvectors of both $A$ and $B$ are
$$ e_1,\ldots,e_{n-1},\frac{Du}{|Du|} $$
where $\{e_1,\ldots,e_{n-1}\}$ is an orthogonal basis of the orthogonal complement of $\operatorname{span}\{Du\}$. 
If $Du=0$, then $A$ and $B$ are well defined and equal to $I$, because
$\vartheta_a(0)=0=\vartheta_b(0)$.
We conclude that the eigenvalues of $A$ and $B$ are
$$ 1,\ldots,1,1+\vartheta_a(|Du|) 
\quad\text{and}\quad
1,\ldots,1,1+\vartheta_b(|Du|), $$
respectively.
The assumption \eqref{eq:Ellipticity-bounds} guarantees that $A$ and $B$ are uniformly elliptic in $\Om$.

The following Lemma is elementary yet useful to state for record. 
We noticed that it appears implicitly in \cite{Grisvard1985}*{Theorem 3.1.1.1} and \cite{Peng}, and undoubtedly in many other instances in the literature.

\begin{lemma} \label{lem:Divergence-structure-observation}
If $X\in C^2(\Om;\Rn)$, then
\begin{align*}
    \diverg((DX-\tr(DX)I)X)
    =\la DX,DX^\intercal\ra -(\tr(DX))^2
\end{align*}
everywhere in $\Om$.
\end{lemma}

\begin{proof} 
The proof is a direct calculation. Indeed,
notice that if $M\in C^1(\Om;\R^{n\times n})$ and $V\in C^1(\Om;\Rn)$, then
\begin{align*}
    \diverg(MV)
    &=
    \sum_{i=1}^n\frac{\partial}{\partial x_i}\Big(\sum_{j=1}^nM_{ij}V_j\Big) \\
    &=
    \sum_{i,j=1}^n\Big(\frac{\partial M_{ij}}{\partial x_i}V_j+M_{ij}\frac{\partial V_j}{\partial x_i}\Big) \\
    &=
    \la\diverg(M^\intercal),V\ra+\la M^\intercal,DV\ra.
\end{align*}
In the above display we use the convention that the divergence of the matrix $M^\intercal$ is a vector whose components are divergences of the rows of $M^\intercal$.

In our case $M=DX-\tr(DX)I$ and $V=X$. Then
\begin{align*}
    \diverg((DX-\tr(DX)I)X)
    &=
    \la \diverg(DX^\intercal-\tr(DX)I),X\ra 
    +\la DX^\intercal,DX\ra -(\tr(DX))^2 \\
    &=
    \la DX^\intercal,DX\ra -(\tr(DX))^2
\end{align*}
because for each $i=1,\ldots,n$
\begin{align*}
    \big(\diverg(DX^\intercal-\tr(DX)I)\big)_i
    &=
    \sum_{j=1}^n\frac{\partial(DX)_{ji}}{\partial x_j} -\sum_{j=1}^n\frac{\partial (\tr(DX)\delta_{ij})}{\partial x_j} \\
    &=
    \sum_{j=1}^n\frac{\partial}{\partial x_j}\Big(\frac{\partial X_j}{\partial x_i}\Big) -\frac{\partial}{\partial x_i}\Big(\sum_{k=1}^n\frac{\partial X_k}{\partial x_k}\Big) \\
    &=0.
\end{align*}
The proof is finished.
\end{proof}

\begin{remark}
It is easy to verify that Lemma \ref{lem:Divergence-structure-observation} holds even more generally. Namely, if $X\in C^2(\Om;\Rn)$ and $Y\in C^1(\Om;\Rn)$ then 
\begin{align*}
    \diverg((DX-\tr(DX)I)Y)
    =\la DX,DY^\intercal\ra -\tr(DX)\tr(DY)
\end{align*}
everywhere in $\Om$.
\end{remark}

\begin{remark} \label{rem:Recovery-of-CM}
If we set $X=V_a$ in Lemma \ref{lem:Divergence-structure-observation}, we recover the inequality \eqref{eq:CM-pointwise-differential-inequality} of Cianchi and Maz'ya. Indeed,
\begin{equation} \label{eq:Computation-for-CM}
\begin{aligned}
    &\diverg\big((DV_a-\tr(DV_a)I)V_a\big) \\
    &=
    \diverg\Big(\big(a(|Du|)\big)^2
    \Big(I+\vartheta_a(|Du|)\frac{Du\otimes Du}{|Du|^2}\Big)D^2uDu\Big) \\
    &\quad
    -\diverg\Big(
    \big(a(|Du|)\big)^2
    \tr\Big(\Big(I+\vartheta_a(|Du|)\frac{Du\otimes Du}{|Du|^2}\Big)D^2u\Big)Du\Big) \\
    &=
    \diverg\big(
    \big(a(|Du|)\big)^2(D^2uDu-\Delta uDu)\big) \\
    &\quad
    +
    \diverg\Big(
    \big(a(|Du|)\big)^2\frac{\vartheta_a(|Du|)}{|Du|^2}
    \big(
    (Du\otimes Du)D^2uDu-\tr((Du\otimes Du)D^2u)Du
    \big)
    \Big) \\
    &=
    \diverg\big(
    \big(a(|Du|)\big)^2(D^2uDu-\Delta uDu)\big).
\end{aligned}     
\end{equation}
In the last equality of the above display we used the fact that
$$ (Du\otimes Du)D^2uDu=\il uDu=\tr((Du\otimes Du)D^2u)Du $$
for any sufficiently smooth function $u$.
Moreover, the estimate \eqref{eq:Lower-bound} in the proof of Lemma \ref{lem:Lower-bound-for-product-of-matrix-and-transpose} implies that
\begin{equation} \label{eq:Estimate-from-below-CM}
    \la DV_a,DV_a^\intercal\ra = \big(a(|Du|)\big)^2\la AD^2u,(AD^2u)^\intercal\ra
    \geq \big(\min\{1,1+i_a\}\big)^2\big(a(|Du|)\big)^2|D^2u|^2.
\end{equation}
Finally notice that
\begin{equation} \label{eq:Final-observation-CM}
   \tr(DV_a)=\diverg(V_a)=\diverg\big(a(|Du|)Du\big). 
\end{equation}
As we combine \eqref{eq:Computation-for-CM}, \eqref{eq:Estimate-from-below-CM} and \eqref{eq:Final-observation-CM} with Lemma \ref{lem:Divergence-structure-observation}, we get an elementary proof of \eqref{eq:CM-pointwise-differential-inequality}. Moreover, notice that in the special case when $a\equiv1$ the inequality \eqref{eq:Estimate-from-below-CM} improves to an equality and we recover the basic identity \eqref{eq:Basic-identity-for-Laplacian}.
\end{remark}

For the statement of the next lemma we set
$$ \theta(t):=\frac{1+\vartheta_a(t)}{1+\vartheta_b(t)}. $$
Note that $\theta$ depends on both $a$ and $b$. In addition we denote
$$ i_\theta:=\inf_{t\geq0} \theta(t)\quad\text{and}\quad
s_\theta:=\sup_{t\geq 0}\theta(t). $$

\begin{lemma} \label{lem:Inequality}
Suppose that $a,b\colon[0,\infty)\to(0,\infty)$ are functions as described above and
\begin{equation} \label{eq:Cordes-for-AwrtB}
    0<i_\theta\leq s_\theta<\frac{2(n-1)}{n-2}.
\end{equation}
Then there exist constants $c=c(n,i_\theta,s_\theta,i_b,s_b)>0$ and $C=C(n,i_\theta,s_\theta)>0$ such that
\begin{equation}
    c|DV_b|^2\leq
    \diverg((DV_b-\tr(DV_b)I)V_b)
    +C\Big(\frac{b(|Du|)}{a(|Du|)}\Big)^2\big(\diverg(V_a)\big)^2.
\end{equation}
\end{lemma}

\begin{proof}
Let $A$ and $B$ be as in \eqref{eq:AB}.
We compute that
\begin{align*}
    B^{-1}A
    &=I+\frac{\vartheta_a(|Du|)-\vartheta_b(|Du|)}{1+\vartheta_b(|Du|)}\frac{Du\otimes Du}{|Du|^2}. 
\end{align*}
Then
\begin{align*}
    \la B^{-1}A,(B^{-1}A)^\intercal\ra 
    &=n-1+(\theta(|Du|))^2
\end{align*}
and
\begin{align*}
    \big(\tr(B^{-1}A)\big)^2
    &=\big(n-1+\theta(|Du|)\big)^2.
\end{align*}
We conclude that $A$ satisfies Cordes condition with respect to $B$ if and only if
\begin{equation} \label{eq:PointwiseCordesofAwrtB}
    (n-1+\delta)\big(n-1+(\theta(|Du|))^2\big)
    \leq\big(n-1+\theta(|Du|)\big)^2
\end{equation}
holds for some $0<\delta<1$. 
In other words, we need that
\begin{align} \label{eq:Upper-bound-for-delta}
    \delta\leq \frac{\big(2(n-1)-(n-2)\theta(|Du|)\big)\theta(|Du|)}{n-1+(\theta(|Du|))^2}.
\end{align}
Indeed, under the condition \eqref{eq:Cordes-for-AwrtB}, the right hand side of \eqref{eq:Upper-bound-for-delta} is uniformly positive and
we may find such $\delta=\delta(n,i_\theta,s_\theta)>0$.

We apply Corollary \ref{cor:General-Cordes} with $M=D^2u$ to obtain
\begin{equation} \label{eq:GeneralFormofCordesUsed}
    c\la BD^2u,(BD^2u)^\intercal \ra \leq
    \la BD^2u,(BD^2u)^\intercal \ra-\big(\tr(BD^2u)\big)^2
    +C\la A,D^2u\ra
\end{equation}
for some $c=c(n,i_\theta,s_\theta)>0$ and $C=C(n,i_\theta,s_\theta)>0$. 

By Lemma \ref{lem:Lower-bound-for-product-of-matrix-and-transpose}
\begin{equation} \label{eq:LowerBoundforProductofMatrixandTransposeUsed}
    \la BD^2u,(BD^2u)^\intercal \ra\geq \Big(\frac{\min\{1,1+i_b\}}{\max\{1,1+s_b\}}\Big)^2|BD^2u|^2.
\end{equation}
The estimates \eqref{eq:GeneralFormofCordesUsed} and \eqref{eq:LowerBoundforProductofMatrixandTransposeUsed} yield that that
\begin{equation} \label{eq:PlainConclusionofCordes}
    c|BD^2u|^2 \leq
    \la BD^2u,(BD^2u)^\intercal \ra-\big(\tr(BD^2u)\big)^2
    +C\la A,D^2u\ra
\end{equation}
where $c=c(n,i_\theta,s_\theta,i_b,s_b)>0$ and $C=C(n,i_\theta,s_\theta)>0$. 
Multiplying the estimate \eqref{eq:PlainConclusionofCordes} by $\big(b(|Du|)\big)^2$ yields that
\begin{align*} 
    c|D(b(|Du|)Du)|^2
    &\leq
    \la D(b(|Du|)Du),(D(b(|Du|)Du))^\intercal \ra-\big(\tr(D(b(|Du|)Du))\big)^2 \\
    &\quad
    +C\big(b(|Du|)\la A,D^2u\ra\big)^2.
\end{align*}
As we employ Lemma \ref{lem:Divergence-structure-observation} and note that 
\begin{equation}
    \diverg(a(|Du|)Du)
    =a(|Du|)\la A,D^2u\ra,
\end{equation}
the desired estimate follows immediately.
\end{proof}

\begin{corollary} \label{cor:Inequality-with-arbitrary-vectorfield}
Let $W\in C^2(\Om;\Rn)$. Under the assumptions of Lemma \ref{lem:Inequality}, 
\begin{align*}
    c|DV_b|^2
    &\leq
    \diverg\big((D(V_b-W)-\tr(D(V_b-W))I)(V_b-W)\big) \\
    &\quad
    +C\Big(|DW|^2+\Big(\frac{b(|Du|)}{a(|Du|)}\Big)^2\big(\diverg(V_a)\big)^2\Big)
\end{align*}
where $c=c(n,i_\theta,s_\theta,i_b,s_b)>0$ and $C=C(n,i_\theta,s_\theta)>0$.
\end{corollary}

Note that Corollary \ref{cor:Inequality-with-arbitrary-vectorfield} reduces to Lemma \ref{lem:Inequality} if $W\equiv0$.

\begin{proof}[Proof of Corollary \ref{cor:Inequality-with-arbitrary-vectorfield}]
Let us fix $W\in C^2(\Om;\Rn)$.
Write $V_b=X+W$, where $X:=V_b-W$. By Lemma \ref{lem:Inequality} and Lemma \ref{lem:Divergence-structure-observation}
\begin{equation} \label{eq:InequalitywithArbitraryVectorFieldFirstStep}
\begin{aligned}
    c|DV_b|^2
    &\leq
    \la DV_b,(DV_b)^\intercal \ra-\big(\tr(DV_b)\big)^2
    +C\Big(\frac{b(|Du|)}{a(|Du|)}\Big)^2\big(\diverg(V_a)\big)^2 \\
    &=
    \diverg\big((DX-\tr(DX)I)X\big) \\
    &\quad
    +2\big(\la DX,(DW)^\intercal\ra-\tr(DX)\tr(DW)\big) 
    +\la DW,(DW)^\intercal\ra-\big(\tr(DW)\big)^2 \\
    &\quad
    +C\Big(\frac{b(|Du|)}{a(|Du|)}\Big)^2\big(\diverg(V_a)\big)^2.
\end{aligned}    
\end{equation}
We apply Young's inequality to estimate
\begin{equation} \label{eq:YoungtoEstimateDxtimesDW}
\begin{aligned}
    &2\big(\la DX,(DW)^\intercal\ra-\tr(DX)\tr(DW)\big)
    +\la DW,(DW)^\intercal\ra-\big(\tr(DW)\big)^2 \\
    &=
    \la DV_b,(DW)^\intercal\ra-\tr(DV_b)\tr(DW)
    -\la DW,(DW)^\intercal\ra+\big(\tr(DW)\big)^2.\\
    &\leq \frac{c}{2}|DV_b|^2+C|DW|^2,
\end{aligned}
\end{equation}
where $C=C(n,c)>0$.
The desired inequality now follows from \eqref{eq:InequalitywithArbitraryVectorFieldFirstStep} and \eqref{eq:YoungtoEstimateDxtimesDW}.
\end{proof}

\section{A boundary identity and its corollaries} \label{sec:Boundary-Identity}

This section is based on \cite{Grisvard1985}*{Section 3.1.1}.
Let $\Om\subset\Rn$ be a bounded domain such that the boundary $\partial\Om$ is smooth. We consider $\partial\Om$ as an $(n-1)$-dimensional smooth submanifold of $\Rn$.
For each boundary point $x\in\partial\Om$, let $\tau_1,\ldots,\tau_{n-1}$ be unit vectors that form an orthonormal basis of the tangent space $T_x\partial\Om$.
Let $\nu\in\Rn$ denote the outward unit normal of $\partial\Om$. Then
$\{\tau_1,\ldots,\tau_{n-1},\nu\}$
forms an orthonormal basis of $\Rn$ at each boundary point. Due to the smoothness of $\partial\Om$, we can assume that these basis vectors are smooth functions $\partial\Om\to\Rn$.

Let $X\in C^1(\partial\Om;\Rn)$ be a vector field.
We write
\begin{equation*} 
   X=X_T+\la X,\nu\ra\nu, 
\end{equation*}
where
\begin{equation} \label{eq:XT}
    X_T:=
    X-\la X,\nu\ra\nu=
    \sum_{i=1}^{n-1}\la X,\tau_i\ra\tau_i
\end{equation}
denotes the part of $X$ that is tangential to the boundary $\partial\Om$.
The tangential divergence of $X$ on the boundary $\partial\Om$ is given by
$$ \diverg_T(X)=\sum_{i=1}^{n-1}\la \frac{\partial X}{\partial s_i},\tau_i\ra $$
where $\frac{\partial}{\partial s_i}$ 
denotes the directional derivative with respect to $\tau_i$.
Let $f\in C^1(\partial\Om)$ be a function. Similarly to \eqref{eq:XT}, the tangential gradient of $f$ on the boundary $\partial\Om$ is given by 
$$ D_Tf=\sum_{i=1}^{n-1}\frac{\partial f}{\partial s_i}\tau_i. $$

\begin{lemma}[\cite{Grisvard1985}*{Equation (3,1,1,8)}] \label{lem:Boundary-identity}
Let $\Om\subset\Rn$ be a bounded domain with a smooth boundary $\partial\Om$. 
If $X\in C^1(\partial\Om;\Rn)$, then
\begin{align*}
    \la DX\cdot X-\diverg(X)X,\nu\ra
    &=
    \la X_T,D_T\la X,\nu\ra\ra-\la X,\nu\ra\diverg_T(X_T) \\
    &\quad
    +\mathcal{B}(X_T,X_T)+\la X,\nu\ra^2\tr(\mathcal{B})
\end{align*}
on the boundary $\partial\Om$. Here $\mathcal{B}$ denotes the second fundamental form of $\partial\Om$.
\end{lemma}

\begin{corollary} \label{cor:Estimate-for-flow-of-X-across-boundary}
Under the assumptions of Lemma \ref{lem:Boundary-identity}, if in addition $X_T=0$ on the boundary $\partial\Om$, then
\begin{align*}
    \big|\la DX\cdot X-\diverg(X)X,\nu\ra\big|
    &\leq C|\mathcal{B}||X|^2
\end{align*}
on the boundary $\partial\Om$, where $C=C(n)>0$.
\end{corollary}

\begin{corollary} \label{cor:Estimate-for-flow-of-X-across-boundary-convex}
Under the assumptions of Lemma \ref{lem:Boundary-identity}, if in addition $\Om$ is convex and $X_T=0$ on the boundary $\partial\Om$, then
\begin{align*}
    \la DX\cdot X-\diverg(X)X,\nu\ra
    \leq 0
\end{align*}
on the boundary $\partial\Om$.
\end{corollary}

\section{Weighted trace inequality} \label{sec:BoundaryEstimates}

In this section we state a weighted trace inequality from \cite{Cianchi2019} which is a crucial part of the proof of Theorem \ref{thm:Global-estimate}. We also discuss about the meaning of the boundary regularity assumption \eqref{eq:Smallness-of-the-limit-of-Kquantity}.  
All the results in this section are from \cites{Cianchi2018,Cianchi2019}. 

We assume that $\Om\subset\Rn$ is a bounded Lipschitz domain
with diameter $d_\Om>0$ and Lipschitz constant $L_\Om>0$. Suppose in addition that $\partial\Om\in W^{2,1}$. Recall that $\mathcal{K}_{\Om}\colon(0,1)\to[0,\infty]$ is defined by
\begin{align} \label{eq:Kquantity}
    \mathcal{K}_{\Om}(r):=\sup_{x\in\partial\Om}
    \sup_{E\subset\partial\Om\cap B_r(x)}
    \frac{\int_E|\mathcal{B}|d\mathcal{H}^{n-1}}{\operatorname{cap}_{B_1(x)}(E)}
\end{align}
where $\mathcal{B}$ denotes the weak second fundamental form of $\partial\Om$ and $\operatorname{cap}_{B_1(x)}(E)$ denotes the capacity of a set $E$ relative to a ball $B_1(x)$.

The following Theorem is a consequence of the Adams' potential embedding theorem, see \cite{Adams1996}*{Theorem 7.2.1}. Here and in similar occurrences in what follows, the dependence of a constant on $L_\Om$ and $d_\Om$ is understood just via an upper bound for them.

\begin{lemma}[\cite{Cianchi2019}*{Lemma 3.5}] \label{lem:Weighted-trace-inequality}
Suppose that $x_0\in\partial\Om$ and let $r>0$ be small. 
There exists a constant $C=C(n,L_\Om,d_\Om)>0$ such that
\begin{align} \label{eq:Weighted-trace-inequality}
    \int_{\partial\Om\cap B_r(x_0)}|v|^2|\mathcal{B}|d\mathcal{H}^{n-1}
    \leq C\mathcal{K}_{\Om}(r)
    \int_{\Om\cap B_r(x_0)}|Dv|^2dx
\end{align}
holds for any $v\in C^1_0(\overline{\Om}\cap B_r(x_0))$.
\end{lemma}

Note that if $\mathcal{K}_{\Om}(r)=\infty$, then the inequality \eqref{eq:Weighted-trace-inequality} is trivially true. Obviously we need $\mathcal{K}_{\Om}(r)<\infty$. Moreover, we will need that $\mathcal{K}_\Om(r)$ is sufficiently small when $r\to0$, which is the condition \eqref{eq:Smallness-of-the-limit-of-Kquantity}.

By \cite{Cianchi2018}, in order to guarantee that $\mathcal{K}_{\Om}(r)$ is finite, it suffices to assume that the 
weak second fundamental form $\mathcal{B}$ belongs to a weak Lebesgue space (for $n\geq3$) or to a weak Zygmund space (for $n=2$).

For the definition of the weak Lebesgue space and weak Zygmund space, we need to introduce some preliminary concepts. 
Denote $\mu:=\mathcal{H}^{n-1}\big|_{\partial\Om}$. That is, $\mu$ is the restriction of the $(n-1)$-dimensional Hausdorff measure to the boundary $\partial\Om$, so that $\mu$ is a natural measure on $\partial\Om$.

The distribution function of ($\mu$-measurable) function $\psi\colon\partial\Om\to\R$ is
\begin{equation}
    \mu_\psi(\lambda)=\mu(\{x\in\partial\Om:|\psi(x)|>\lambda\})
    \quad\text{for all }\lambda>0.
\end{equation}
Given the distribution function $\mu_\psi$, we define the decreasing rearrangement of $\psi$, denoted by $\psi^*$, as
$$ \psi^*(t)
:=\sup\{\lambda>0:\mu_\psi(\lambda)>t\}
\quad\text{for all }t>0.$$
It holds that
\begin{equation}\label{eq:Rearrangement-integral}
    \int_{\partial\Om}|\psi|d\mu
    =\int_0^{\mu(\partial\Om)} \psi^*(s)ds.
\end{equation}

Let $q>1$. The weak Lebesgue space $L^{q,\infty}(\partial\Om)$ is the space of all $\mu$-measurable functions $\psi$ such that
\begin{equation}
    \Vert \psi\Vert_{L^{q,\infty}(\partial\Om)}
    :=\sup_{0<s<\mu(\partial\Om)}s^{\frac{1}{q}-1}\int_0^s \psi^*(t)dt<\infty
\end{equation}
The weak Zygmund space $L^{1,\infty}\log L(\partial\Om)$ is the space of all $\mu$-measurable functions $\psi$ such that
\begin{equation}
    \Vert \psi\Vert_{L^{1,\infty}\log L (\partial\Om)}
    :=\sup_{0<s<\mu(\partial\Om)}\log\Big(1+\frac{C}{s}\Big)\int_0^s \psi^*(t)dt<\infty.
\end{equation}
By \cite{Cianchi2018}*{Proof of Theorem 2.4}, see also  Lemmas 3.5 and 3.7 in \cite{Cianchi2019},
\begin{equation}
\begin{aligned}
    \mathcal{K}_{\Om}(r)\lesssim
    \begin{cases}
    \displaystyle{
    \sup_{x\in\partial\Om}\|\mathcal{B}\|_{L^{n-1,\infty}(\partial\Om\cap B_r(x))}}
    \quad&\text{if }n\geq 3, \\
    \displaystyle{
    \sup_{x\in\partial\Om}\|\mathcal{B}\|_{L^{1,\infty}\log L(\partial\Om\cap B_r(x))}}
    \quad&\text{if }n=2,
    \end{cases}
\end{aligned}    
\end{equation}
up to some positive constant depending on $n$, $d_\Om$ and $L_\Om$.
In particular, if $\mathcal{B}\in L^{n-1,\infty}(\partial\Om)$ for $n\geq 3$ or $\mathcal{B}\in L^{1,\infty}\log L(\partial\Om)$ for $n=2$, then $\mathcal{K}_{\Om}(r)<\infty$ for all $r\in(0,1)$ sufficiently small.
The smallness assumption \eqref{eq:Smallness-of-the-limit-of-Kquantity} for $\mathcal{K}_\Om$ is then certainly satisfied if
\begin{equation} \label{eq:CM-assumption-ngeq3}
    \lim_{r\to0}
    \Big(\sup_{x\in\partial\Om}\|\mathcal{B}\|_{L^{n-1,\infty}(\partial\Om\cap B_r(x))}\Big)
    <\mathcal{K}'_0
    \quad\text{if }n\geq 3,
\end{equation}
or
\begin{equation} \label{eq:CM-assumption-neq2}
    \lim_{r\to0}
    \Big(\sup_{x\in\partial\Om}\|\mathcal{B}\|_{L^{1,\infty}\log L(\partial\Om\cap B_r(x))}\Big)
    <\mathcal{K}'_0
    \quad\text{if }n=2,
\end{equation}
with a suitable $\mathcal{K}'_0=\mathcal{K}'_0(n,p,d_\Om,L_\Om)>0$.
The assumptions \eqref{eq:CM-assumption-ngeq3} and \eqref{eq:CM-assumption-neq2} were the main boundary regularity assumptions used in \cite{Cianchi2018}*{Theorem 2.4}. The quantity $\mathcal{K}_\Om$ was introduced in the later work \cite{Cianchi2019}. 
By \cite{Cianchi2019}*{Remark 2.5}, if $\partial\Om\in C^2$, then 
$$ \lim_{r\to0}\mathcal{K}_\Om(r)=0, $$
and thus \eqref{eq:Smallness-of-the-limit-of-Kquantity} holds trivially.

\section{Regularized case} \label{sec:Regularization}

In this section we prove a regularized version of our main theorem.
Let $\Om\subset\Rn$ be a bounded domain with smooth boundary $\partial\Om$. 
We denote the diameter of $\Om$ by $d_\Om$ and the Lipschitz constant of the boundary $\partial\Om$ by $L_\Om$. 

Let $\varphi\in C^\infty(\overline{\Om})$. For $1<p<\infty$ and $0<\epsilon<1$, consider the problem of minimizing the regularized $p$-energy functional
$$  \int_\Om(|Dv|^2+\epsilon)^{p/2}dx $$
among $v\in \varphi+W^{1,p}_0(\Om)$. By the standard methods of calculus of variations, there exists a unique minimizer $u\in \varphi+W^{1,p}_0(\Om)$. Moreover, the minimizer $u$ is a weak solution to the Dirichlet problem
\begin{equation} \label{eq:Dirichlet-problem-regularized}
\begin{cases}
\begin{aligned}
\diverg(a(|Du|)Du)=0 &\quad\text{in }\Om; \\
u=\varphi &\quad\text{on }\partial \Om,
\end{aligned}
\end{cases}
\end{equation}
where $a\colon[0,\infty)\to(0,\infty)$ is given by
$$ a(t):=\big(t^2+\epsilon\big)^{\frac{p-2}{2}} 
\quad\text{for all }t\geq 0. $$
By the classical regularity theory \cite{GilbargTrudninger}, we have $u\in C^\infty(\overline{\Om})$.

Denote $V_a:=a(|Du|)Du$. Then $V_a$ is of divergence free by the PDE in \eqref{eq:Dirichlet-problem-regularized}. Let $\beta>-1$ and consider the vector field $V_b=b(|Du|)Du$ where
$$ b(t):=\big(t^2+\epsilon\big)^{\frac{\beta}{2}} 
\quad\text{for all }t\geq 0. $$ 
The functions $a$ and $b$ satisfy the assumption \eqref{eq:Ellipticity-bounds} with
\begin{align*}
    i_a=\min\{p-2,0\}\quad\text{and}\quad s_a=\max\{p-2,0\},
\end{align*}
and
\begin{align*}
    i_b=\min\{\beta,0\}\quad\text{and}\quad s_b=\max\{\beta,0\}.
\end{align*}
Moreover,
\begin{align*}
    \theta(t)=
    \frac{1+\vartheta_a(t)}{1+\vartheta_b(t)}=
    \frac{(p-1)t^2+\epsilon}{(\beta+1)t^2+\epsilon}
\end{align*}
and
\begin{align*}
    i_\theta=\min\Big\{\frac{p-1}{\beta+1},1\Big\}\quad\text{and}\quad 
    s_\theta
    =\max\Big\{\frac{p-1}{\beta+1},1\Big\}.
\end{align*}
In particular \eqref{eq:Cordes-for-AwrtB} is satisfied if and only if 
\begin{align*}
    \frac{p-1}{\beta+1}<\frac{2(n-1)}{n-2},
\end{align*}
or equivalently
\begin{align} \label{eq:Condition-for-beta}
    \beta>-1+\frac{(n-2)(p-1)}{2(n-1)}.
\end{align}

The following auxiliary lemma can be applied to derive either local or global estimates.

\begin{lemma} \label{lem:Auxiliary-Lemma}
Let $1<p<\infty$ and let $u\in C^\infty(\overline{\Om})$ be a solution to \eqref{eq:Dirichlet-problem-regularized}. If
$$ \beta>-1+\frac{(n-2)(p-1)}{2(n-1)}, $$
then there exists a constant $C=C(n,p,\beta)>0$ such that 
\begin{equation} \label{eq:Local-estimate-for-interior-and-boundary}
\begin{aligned}
    \int_\Om|DV_b|^2\eta^2dx
    &\leq 
    C\int_{\partial\Om}\la (D(V_b-W)-\tr(D(V_b-W))I)(V_b-W),\nu\ra\eta^2d\mathcal{H}^{n-1} \\
    &\quad
    +C\int_\Om|V_b-W|^2|D\eta|^2dx
    +C\int_{\Om}|DW|^2\eta^2dx 
\end{aligned}    
\end{equation}
for any $\eta\in C^\infty(\Rn)$ and for any $W\in C^2(\Om;\Rn)$.
\end{lemma}

\begin{proof}
Fix $\eta\in C^\infty(\Rn)$ and $W\in C^2(\Om;\Rn)$. Denote $X:=V_b-W$.
By Corollary \ref{cor:Inequality-with-arbitrary-vectorfield}
\begin{equation} \label{eq:IntegratedInequalitywithArbitraryVectorField}
\begin{aligned}
    c\int_\Om|DV_b|^2\eta^2dx
    &\leq     
    \int_\Om\Big(\diverg\big((DX-\tr(DX)I)X \big)\Big)\eta^2dx \\
    &\quad
    +C\int_\Om|DW|^2\eta^2dx +C\int_\Om\Big(\frac{b(|Du|)}{a(|Du|)}\Big)^2(\diverg(V_a))^2\eta^2dx.
\end{aligned}
\end{equation}
where
\begin{equation} \label{eq:divVa-vanishes-by-PDE}
   \diverg(V_a)=\diverg(a(|Du|)Du)=0 
\end{equation}
by the PDE in \eqref{eq:Dirichlet-problem-regularized}.

By Gauss divergence theorem
\begin{equation} \label{eq:DivergenceTheoremAppliedonDivergenceTerm}
\begin{aligned}
    &\int_\Om\Big(\diverg\big((DX-\tr(DX)I)X \big)\Big)\eta^2dx \\
    &\quad
    =
    \int_{\partial\Om}\la (DX-\tr(DX)I)X,\nu\ra \eta^2d\mathcal{H}^{n-1} 
    -2\int_\Om\la (DX-\tr(DX)I)X,D\eta\ra\eta dx \\
    &\quad
    \leq 
    \int_{\partial\Om}\la (DX-\tr(DX)I)X,\nu\ra \eta^2d\mathcal{H}^{n-1} 
    +C\int_\Om|DX||X||D\eta||\eta| dx.
\end{aligned}
\end{equation}
We combine \eqref{eq:IntegratedInequalitywithArbitraryVectorField}, \eqref{eq:divVa-vanishes-by-PDE} and \eqref{eq:DivergenceTheoremAppliedonDivergenceTerm} to get
\begin{equation} \label{eq:IntegratedInequalityandGaussCombined}
\begin{aligned}
    c\int_\Om|DV_b|^2\eta^2dx
    &\leq \int_{\partial\Om}\la (DX-\tr(DX)I)X,\nu\ra\eta^2d\mathcal{H}^{n-1} \\
    &\quad
    +C\int_\Om|DX||X||D\eta||\eta|dx \\
    &\quad
    +C\int_{\Om}|DW|\eta^2dx
    +C\int_\Om\Big(\frac{b(|Du|)}{a(|Du|)}\Big)^2(\diverg(V_a))^2\eta^2dx.
\end{aligned}    
\end{equation}
The claim follows from \eqref{eq:IntegratedInequalityandGaussCombined} by using triangle inequality and Young's inequality.
\end{proof}

\begin{remark} \label{rem:Previous-work}
If the support of $\eta$ lies inside $\Om$, then the boundary integral in \eqref{eq:Local-estimate-for-interior-and-boundary} vanishes and we can easily derive local estimates. 
More precisely, given some concentric balls $B_r\subset B_{2r}\subset\subset\Om$, select a cutoff function $\eta\in C^\infty_0(\Rn)$ such that
$$ \spt(\eta)\subset B_{2r},\quad 
\eta\equiv 1\enskip\text{in}\enskip B_r
\quad\text{and}\quad
|D\eta|\leq \frac{10}{r}, $$
and set
$$ W\equiv (V_b)_{B_{2r}}=\fint_{B_{2r}}V_bdx $$
in Lemma \ref{lem:Auxiliary-Lemma}.
Then it follows easily from \eqref{eq:Local-estimate-for-interior-and-boundary} that
\begin{equation}
\begin{aligned}
    \int_{B_r}|DV_b|^2dx
    \leq     
    \frac{C}{r^2}\int_{B_{2r}}|V_b-(V_b)_{B_{2r}}|^2 dx. 
\end{aligned}
\end{equation}
for some $C=C(n,p,\beta)$.
By letting $\epsilon\to0$ we recover Theorem 1.1 from \cite{Sarsa2022}.
\end{remark}

\begin{remark}
Suppose that $\Om\subset\Rn$ is convex. If $1<p<3+\frac{2}{n-2}$, then
Lemma \ref{lem:Auxiliary-Lemma} is applicable.
We select $\eta\equiv1$ and $W=D\varphi$ in Lemma \ref{lem:Auxiliary-Lemma} to find
\begin{equation} \label{eq:Estimate-with-eta-identically-one}
\begin{aligned}
    \int_\Om|D^2u|^2dx
    &\leq 
    C\int_{\partial\Om}\la (D^2u-D^2\varphi)-\tr(D^2u-D^2\varphi))I)(Du-D\varphi),\nu\ra d\mathcal{H}^{n-1} \\
    &\quad
    +C\int_{\Om}|D^2\varphi|^2dx.,
\end{aligned}    
\end{equation}
for some $C=C(n,p)>0$.
Since $D_Tu=D_T\varphi$ on the boundary $\partial \Om$, we can apply Corollary \ref{cor:Estimate-for-flow-of-X-across-boundary-convex} to find that the boundary integral on the right hand side of \eqref{eq:Estimate-with-eta-identically-one} is always nonpositive. Consequently 
\begin{equation} \label{eq:Estimate-convex}
\begin{aligned}
    \int_\Om|D^2u|^2dx
    &\leq 
    C\int_{\Om}|D^2\varphi|^2dx,
\end{aligned}    
\end{equation}
for some $C=C(n,p)>0$.
The estimate \eqref{eq:Estimate-convex} can be derived by using the tools presented in \cite{Manfredi1988}*{Section 2}. See also \cite{Talenti1965}.
\end{remark}

For the proof of the following version of Sobolev's inequality, see \cite{Mazya2011}*{Proof of Theorem 1.4.6/1}. See also \cite{Cianchi2019}*{Lemma 3.4}.

\begin{lemma} \label{lem:Version-of-Sobolev}
Let $\Om\subset\Rn$ be a bounded Lipschitz domain with diameter $d_\Om>0$ and Lipschitz constant $L_\Om>0$. Then for any $\sigma>0$ we can find $C=C(\sigma,n,d_\Om,L_\Om)>0$ such that
\begin{align*}
    \int_\Om|v|^2dx
    &\leq \sigma\int_\Om|Dv|^2dx+C\Big(\int_\Om|v|dx\Big)^{2}
\end{align*}
for all $v\in W^{1,2}(\Om)$.
\end{lemma}

\begin{proposition} \label{prop:Global-estimate}
Let $\Om\subset\Rn$ be a bounded, smooth domain with diameter $d_\Om>0$ and Lipschitz constant $L_\Om>0$. 
Suppose that $1<p<3+\frac{2}{n-2}$ and let $u\in C^\infty(\overline{\Om})$ solve \eqref{eq:Dirichlet-problem-regularized}. 
There exists a constant $\mathcal{K}_0=\mathcal{K}_0(n,p,d_\Om,L_\Om)>0$ such that if
\begin{equation} \label{eq:Upper-bound-for-K-quantity}
    \mathcal{K}_\Om(r)\leq \mathcal{K}(r) 
\quad \text{for all }r\in(0,1)
\end{equation}
for some function $\mathcal{K}\colon(0,1)\to[0,\infty)$ satisfying
\begin{equation} \label{eq:Smallness-of-the-limit-of-Kquantity-regularized}
    \lim_{r\to 0}\mathcal{K}(r)<\mathcal{K}_0,
\end{equation}
then
\begin{equation}
\begin{aligned}
    \|Du\|_{W^{1,2}(\Om;\Rn)}
    \leq C\Big(\|D\varphi\|_{W^{1,2}(\Om;\Rn)}+\|D\varphi\|_{L^p(\Om;\Rn)}+\epsilon\Big),
\end{aligned}
\end{equation}
for some $C=C(n,p,d_\Om,L_\Om,\mathcal{K})$.
\end{proposition}

\begin{proof}
Let $\{B_i=B(x_i,r_i)\}_{i=1}^N$ be a covering of $\Om$ such that either $x_i\in\partial\Om$ or $B_i\subset\subset\Om$. 
Such covering can be selected so that the multiplicity of the overlapping balls of $\{B_i\}_{i=1}^N$ only depends on $n$.
Let $\eta_i\in C^\infty_0(B_i)$ be such that $|D\eta_i|\leq \frac{C}{r_i}$ for some absolute constant $C>0$ and $\{\eta_i^2\}_{i=1}^N$ forms a partition of unity subordinate to $\{B_i\}_{i=1}^N$. 

Fix $\eta_i$ for some $i=1,\ldots,N$. We consider two cases separately; either $B_i\subset\subset\Om$ or $x_i\in\partial\Om$.

For $B_i\subset\subset\Om$, we select $\eta=\eta_i$ and $W\equiv 0$ in \eqref{eq:Local-estimate-for-interior-and-boundary} to get
\begin{equation} \label{eq:LocalEstimateForInterior}
\begin{aligned}
    \int_\Om|D^2u|^2\eta_i^2dx
    \leq C\int_\Om|Du|^2|D\eta_i|^2dx,
\end{aligned}    
\end{equation}
for some $C=C(n,p)>0$.

For $x_i\in\partial\Om$, we select $\eta=\eta_i$ and $W=D\varphi$ in \eqref{eq:Local-estimate-for-interior-and-boundary} to get
\begin{equation} \label{eq:LocalEstimateForBoundary-starting-point}
\begin{aligned}
    &\int_\Om|D^2u|^2\eta_i^2dx \\
    &\leq 
    C\int_{\partial\Om}\la (D^2u-D^2\varphi)-\tr(D^2u-D^2\varphi)I)(Du-D\varphi),\nu\ra\eta_i^2d\mathcal{H}^{n-1} \\
    &\quad\quad
    +C\int_\Om|Du-D\varphi|^2|D\eta_i|^2dx
    +C\int_{\Om}|D^2\varphi|^2\eta_i^2dx 
\end{aligned}    
\end{equation}
for some $C=C(n,p)>0$. Since $D_Tu=D_T\varphi$ on the boundary $\partial\Om$, we can apply Corollary \ref{cor:Estimate-for-flow-of-X-across-boundary} to estimate the boundary integral on the right hand side of  \eqref{eq:LocalEstimateForBoundary-starting-point}. We obtain the estimate
\begin{equation} \label{eq:LocalEstimateForBoundary}
\begin{aligned}
    \int_\Om|D^2u|^2\eta_i^2dx
    &\leq C\int_{\partial\Om}|\mathcal{B}||Du-D\varphi|^2\eta_i^2d\mathcal{H}^{n-1}
    +C\int_\Om|Du-D\varphi|^2|D\eta_i|^2dx \\
    &\quad
    +C\int_{\Om}|D^2\varphi|^2\eta_i^2dx
\end{aligned}    
\end{equation}
for some $C=C(n,p)>0$.
By Lemma \ref{lem:Weighted-trace-inequality}, 
for small $r_i>0$
\begin{equation} \label{eq:EstimateForBoundaryIntegral}
\begin{aligned}
    &\int_{\partial\Om}|\mathcal{B}||Du-D\varphi|^2\eta_i^2d\mathcal{H}^{n-1} \\
    &\leq C\mathcal{K}_\Om(r_i)\int_{\Om}|D((Du-D\varphi)\eta_i)|^2dx \\
    &\leq C\mathcal{K}(r_i)\Big(
    \int_{\Om}|D^2u|^2\eta_i^2dx
    +\int_{\Om}|D^2\varphi|^2\eta_i^2dx
    +\int_{\Om}|Du-D\varphi|^2|D\eta_i|^2dx\Big)
\end{aligned}    
\end{equation}
where $C=C(n,p,d_\Om,L_\Om)>0$. In the last inequality of the above display \eqref{eq:EstimateForBoundaryIntegral} we employed the assumption \eqref{eq:Upper-bound-for-K-quantity}.

Sum the interior estimates \eqref{eq:LocalEstimateForInterior} and the boundary estimates \eqref{eq:EstimateForBoundaryIntegral} over $i=1,\ldots,N$ to obtain
\begin{align*}
    \int_\Om|D^2u|^2dx
    &\leq 
    C\Big(\max_{i}\mathcal{K}(r_i)
    \int_{\Om}|D^2u|^2dx \\
    &\quad
    +(1+\max_{i}\mathcal{K}(r_i))
    \int_{\Om}|D^2\varphi|^2dx \\
    &\quad
    +\frac{(1+\max_{i}\mathcal{K}(r_i))}{(\min_ir_i)^2}\Big(\int_\Om|Du|^2+|D\varphi|^2dx\Big)\Big)
\end{align*}
for some $C=C(n,p,d_\Om,L_\Om)>0$.
By the convergence property \eqref{eq:Smallness-of-the-limit-of-Kquantity-regularized} of the function $\mathcal{K}$, we can find $r'=r'(n,p,d_\Om,L_\Om,\mathcal{K})>0$ such that
$0<r_i<r'$ implies that
$$ \mathcal{K}(r_i)\leq \mathcal{K}_0. $$
Similarly as in the proof of Theorem 3.1 in \cite{Cianchi2019} we can choose the covering $\{B_i\}_{i=1}^N$ such that each $r_i$ is not only bounded from above by $r'$, but also bounded from below by $r''=r''(n,p,d_\Om,L_\Om,\mathcal{K})>0$.
That is, $r''<r_i<r'$ for all $i=1,\ldots,N$.
With such covering $\{B_i\}_{i=1}^N$ we obtain
\begin{equation*}
\begin{aligned}
    (1-C\mathcal{K}_0)\int_\Om|D^2u|^2dx
    &\leq \frac{C(1+\mathcal{K}_0)}{(r'')^2}\Big(
    \int_{\Om}|D^2\varphi|^2+|D\varphi|^2dx
    +\int_\Om|Du|^2dx\Big)
\end{aligned}    
\end{equation*}
for some $C=C(n,p,d_\Om,L_\Om)>0$.
If $\mathcal{K}_0\leq \frac{1}{2C}$, then we can conclude that
\begin{equation*}
\begin{aligned}
    \int_\Om|D^2u|^2dx
    &\leq C\Big(
    \int_{\Om}|D^2\varphi|^2+|D\varphi|^2dx
    +\int_\Om|Du|^2dx\Big).
\end{aligned}    
\end{equation*}
for some $C=C(n,p,d_\Om,L_\Om,\mathcal{K})>0$.

It remains to derive an estimate for $\int_\Om|Du|^2dx$.
If $2\leq p<3+\frac{2}{n-2}$ we can simply use Hölder's inequality to obtain
\begin{align*}
    \int_\Om|Du|^2dx
    &\leq |\Om|^{\frac{p-2}{2p}}\Big(\int_{\Om}|Du|^pdx\Big)^{2/p}.
\end{align*}
If $1<p<2$, we can use Lemma \ref{lem:Version-of-Sobolev} to obtain
\begin{align*}
    \int_\Om|Du|^2dx
    &\leq \sigma\int_\Om|D^2u|^2dx+C\Big(\int_\Om|Du|dx\Big)^{2} \\
    &\leq \sigma\int_\Om|D^2u|^2dx+C|\Om|^{\frac{p-1}{p}}\Big(\int_\Om|Du|^pdx\Big)^{2/p}
\end{align*}
for any $\sigma>0$ and for some $C=C(n,d_\Om,L_\Om,\sigma)>0$. 

In both cases we conclude the estimate
\begin{align} \label{eq:Estimate-with-2norm-replaced-with-pnorm}
    \int_{\Om}|D^2u|^2dx
    +\int_{\Om}|Du|^2dx
    \leq 
    C\Big(\Big(\int_\Om|Du|^pdx\Big)^{2/p}
    +\int_{\Om}|D^2\varphi|^2+|D\varphi|^2dx\Big)
\end{align}
for some $C=C(n,p,d_\Om,L_\Om,\mathcal{K})>0$.
Since $u$ minimizes the regularized $p$-energy,
\begin{equation} \label{eq:First-order-estimate}
\begin{aligned}
    \int_{\Om}|Du|^pdx
    \leq \int_\Om(|Du|^2+\epsilon)^{p/2}dx 
    \leq \int_\Om(|D\varphi|^2+\epsilon)^{p/2}dx.
\end{aligned}    
\end{equation}
The desired estimate follows now from \eqref{eq:Estimate-with-2norm-replaced-with-pnorm} and \eqref{eq:First-order-estimate}.
\end{proof}

\section{Proof of Theorem \ref{thm:Global-estimate}} \label{sec:Proof}

In this section we remove the additional regularity assumptions that were imposed in the previous section.

Let $\Om\subset\Rn$ be a bounded Lipschitz domain such that $\partial\Om\in W^{2,1}$. We denote the diameter of $\Om$ by $d_\Om$ and the Lipschitz constant of the boundary $\partial\Om$ by $L_\Om$.
Recall that the function
$\mathcal{K}_{\Om}\colon(0,1)\to[0,\infty]$, given by
\begin{align} \label{eq:Kquantity-again}
    \mathcal{K}_{\Om}(r)=\sup_{x\in\partial\Om}
    \sup_{E\subset\partial\Om\cap B_r(x)}
    \frac{\int_E|\mathcal{B}|d\mathcal{H}^{n-1}}{\operatorname{cap}_{B_1(x)}(E)},
\end{align}
is assumed to satisfy the smallness condition \eqref{eq:Smallness-of-the-limit-of-Kquantity}. That is,
\begin{equation*} 
    \lim_{r\to 0}\mathcal{K}_{\Om}(r)<\mathcal{K}_0
\end{equation*} 
for a suitable upper bound $\mathcal{K}_0=\mathcal{K}_0(n,p,d_\Om,L_\Om)>0$.

Let $u\in W^{1,p}(\Om)$ be a weak solution to the Dirichlet problem
\begin{equation} \label{eq:Dirichlet-problem-homogeneous}
\begin{cases}
\begin{aligned}
\Delta_pu=0 &\quad\text{in }\Om; \\
u=\varphi &\quad\text{on }\partial \Om,
\end{aligned}
\end{cases}
\end{equation}
where $\varphi\in W^{1,p}(\Om)\cap W^{2,2}(\Om)$.

For the proof of Theorem \ref{thm:Global-estimate}, we approximate the operator $\Delta_p$ and the domain $\Om$ similarly as in \cites{Cianchi2018,Cianchi2019}. For the approximation of the boundary data $\varphi$, we employ the Sobolev extension theorem. 

\begin{lemma} \label{lem:Extension}
If $\Om\subset\Rn$ is a bounded Lipschitz domain, then the extension operator
$$ E\colon W^{1,p}(\Om)\cap W^{2,2}(\Om)
\to
W^{1,p}(\Rn)\cap W^{2,2}(\Rn) $$
is a well-defined bounded linear operator. 
\end{lemma}

\begin{proof}
The proof follows from \cite{Calderon1961}*{Theorem 12}.
\end{proof}

Before we give the proof of Theorem \ref{thm:Global-estimate}, let us explain our approximation procedure in detail. For $\epsilon>0$ small, let
$$ a^\epsilon(t):=\big(t^2+\epsilon\big)^{\frac{p-2}{2}} \quad\text{for all }t\geq0. $$
Let $E\varphi\colon\Rn\to\R$ denote the Sobolev extension of $\varphi$ given by Lemma \ref{lem:Extension}. Take 
$\{\varphi_k\}_{k=1}^\infty\subset C^\infty_0(\Rn)$ such that 
\begin{align} \label{eq:Convergece-of-boundary-data}
   \varphi_k\xrightarrow{k\to\infty}E\varphi
   \quad\text{in }W^{2,2}(\Rn)\cap W^{1,p}(\Rn).
\end{align}
Finally, by \cite{Cianchi2019}*{Lemma 5.2}, we can take a sequence of approximation domains $\{\Om_m\}_{m=1}^{\infty}$ such that
\begin{itemize}
    \item[(a)] $\Om\subset\Om_m$
    \item[(b)] $\partial\Om_m\in C^\infty$
    \item[(c)] $|\Om_m\setminus\Om|\xrightarrow{m\to\infty}0$
    \item[(d)] Hausdorff distance of $\Om$ and $\Om_m$ tends to 0 as $m\to\infty$
    \item[(e)] $d_{\Om_m}\leq Cd_\Om$, where $d_{\Om_m}$ denotes the diameter of $\Om_m$ and $C>0$ is a positive constant
    \item[(f)] $L_{\Om_m}\leq CL_\Om$, where $L_{\Om_m}$ denotes Lipschitz constant of the boundary $\partial\Om_m$ and $C>0$ is a positive constant
    \item[(g)] $\mathcal{K}_{\Om_m}(r)\leq C\mathcal{K}_{\Om}(r)$ for all $r\in(0,r_0)$, where 
    \begin{align*} 
    \mathcal{K}_{\Om_m}(r)=\sup_{x\in\partial\Om_m}
    \sup_{E\subset\partial\Om_m\cap B_r(x)}
    \frac{\int_E|\mathcal{B}_m|d\mathcal{H}^{n-1}}{\operatorname{cap}_{B_1(x)}(E)},
    \end{align*}
    $\mathcal{B}_m$ denotes the second fundamental form of $\partial\Om_m$,
    and $C>0$ and $r_0>0$ are positive constants.
\end{itemize}
In fact, note that the convergence \eqref{eq:Convergece-of-boundary-data} implies 
\begin{align} \label{eq:Convergece-of-boundary-data-in-Om}
   \varphi_k\xrightarrow{k\to\infty}E\varphi
   \quad\text{in }W^{2,2}(\Om_m)\cap W^{1,p}(\Om_m)
\end{align}
for each $m=1,2,\ldots$.

Consider the problem of minimizing of the regularized $p$-energy functional
$$  \int_{\Om_m}(|Dv|^2+\epsilon)^{p/2}dx $$
among $v\in \varphi_k+ W^{1,p}_0(\Om_m)$. By direct method of calculus of variations, there exists a unique minimizer $u^{\epsilon,k,m}\in W^{1,p}(\Om_m)$ that solves the Dirichlet problem
\begin{equation} \label{eq:DirichletProblemHomogeneous-regularized-1}
\begin{cases}
\begin{aligned}
\diverg\big(a^\epsilon(|Du^{\epsilon,k,m}|)
Du^{\epsilon,k,m}\big)=0 &\quad\text{in }\Om_m; \\
u^{\epsilon,k,m}=\varphi_k &\quad\text{on }\partial \Om_m
\end{aligned}
\end{cases}
\end{equation}
in the weak sense.

Standard elliptic regularity theory implies that
$u^{\epsilon,k,m}\in C^\infty(\overline{\Om}_m)$. Therefore, Proposition \ref{prop:Global-estimate} is applicable to $u^{\epsilon,m,k}$, provided that $1<p<3+\frac{2}{n-2}$. 

\underline{Step 1}: Let $\epsilon\to0$. Fix $k,m=1,2,\ldots$. Consider the problem of minimizing the $p$-energy functional
$$  \int_{\Om_m}|Dv|^pdx $$
among $v\in \varphi_k+W^{1,p}_0(\Om_m)$.
By direct method in calculus of variations, there exists a unique minimizer
$u^{k,m}\in W^{1,p}(\Om_m)$ that solves the Dirichlet problem
\begin{equation} \label{eq:DirichletProblemHomogeneous-regularized-2}
\begin{cases}
\begin{aligned}
\Delta_p u^{k,m}=0 &\quad\text{in }\Om_m; \\
u^{k,m}=\varphi_k &\quad\text{on }\partial \Om_m
\end{aligned}
\end{cases}
\end{equation}
in the weak sense. Moreover,
\begin{align} \label{eq:Convergence-of-u-1}
    u^{\epsilon,k,m}\xrightarrow{\epsilon\to 0}u^{k,m}
    \quad\text{in }W^{1,p}(\Om_m).
\end{align}

\underline{Step 2}: Let $k\to\infty$. Fix $m=1,2,\ldots$. Consider the problem of minimizing the $p$-energy functional
$$  \int_{\Om_m}|Dv|^pdx $$
among $v\in E\varphi+W^{1,p}_0(\Om_m)$.
By direct method in calculus of variations, there exists a unique minimizer
$u^{m}\in W^{1,p}(\Om_m)$ that solves the Dirichlet problem
\begin{equation} \label{eq:DirichletProblemHomogeneous-regularized-3}
\begin{cases}
\begin{aligned}
\Delta_p u^{m}=0 &\quad\text{in }\Om_m; \\
u^{m}=E\varphi &\quad\text{on }\partial \Om_m
\end{aligned}
\end{cases}
\end{equation}
in the weak sense. Moreover,
\begin{align} \label{eq:Convergence-of-u-2} 
    u^{k,m}\xrightarrow{k\to\infty}u^{m}
    \quad\text{in }W^{1,p}(\Om_m).
\end{align}

\underline{Step 3}: Let $m\to\infty$. For the final step, we have, by standard methods in calculus of variations that
\begin{align} \label{eq:Convergence-of-u-3}
    u^{m}\xrightarrow{k\to\infty}u
    \quad\text{in }W^{1,p}(\Om),
\end{align}
where $u$ is the solution of \eqref{eq:Dirichlet-problem-homogeneous}.

We conclude that we can reach the $p$-harmonic function $u$ from the regularized version $u^{\epsilon,k,m}$ via the three steps described above. Now we are ready to prove our main theorem.

\begin{proof}[Proof of Theorem \ref{thm:Global-estimate}]
Suppose that $1<p<3+\frac{2}{n-1}$. Let $u\in W^{1,p}(\Om)$ solve \eqref{eq:Dirichlet-problem-homogeneous} and $u^{\epsilon,k.m}\in C^\infty(\overline{\Om}_m)$ solve \eqref{eq:DirichletProblemHomogeneous-regularized-1}, as explained above.
By Proposition \ref{prop:Global-estimate}, together with the properties (e) -- (g) of the approximating domain $\Om_m$, there exists a constant \\
$\mathcal{K}_0=\mathcal{K}_0(n,p,d_\Om,L_\Om)>0$ such that if
\begin{equation*} 
    \lim_{r\to 0}\mathcal{K}_\Om(r)<\mathcal{K}_0
\end{equation*}
then
\begin{equation} \label{eq:Global-estimate-with-uekm}
\begin{aligned}
    \|Du^{\epsilon,k,m}\|_{W^{1,2}(\Om_m;\Rn)}
    \leq C\Big(\|D\varphi_k\|_{W^{1,2}(\Om_m;\Rn)}+\|D\varphi_k\|_{L^p(\Om_m;\Rn)}
    +\epsilon\Big).
\end{aligned}
\end{equation}
for all $\epsilon>0$ and $k,m=1,2,\ldots$ and for some constant $C=C(n,p,d_{\Om},L_{\Om},\mathcal{K}_\Om)>0$.
In particular, notice that the constant $C=C(n,p,d_\Om,L_\Om,\mathcal{K}_\Om)$ is uniform with respect to all the regularization parameters $\epsilon$, $k$ and $m$. In the following, $C$ denotes a positive constant that is allowed to depend on $n$, $p$, $d_\Om$, $L_\Om$ and $\mathcal{K}_\Om$. 

\underline{Step 1}: Let $\epsilon\to0$. For $k,m=1,2,\ldots$ fixed, consider the family $\{Du^{\epsilon,k,m}\}_{0<\epsilon<1}$. By \eqref{eq:Global-estimate-with-uekm},
\begin{equation} \label{eq:Global-estimate-with-uekm-1}
\begin{aligned}
    \|Du^{\epsilon,k,m}\|_{W^{1,2}(\Om_m;\Rn)}
    \leq C\Big(\|D\varphi_k\|_{W^{1,2}(\Om_m;\Rn)}+\|D\varphi_k\|_{L^p(\Om_m;\Rn)}
    +1\Big).
\end{aligned}
\end{equation}
The right hand side of \eqref{eq:Global-estimate-with-uekm-1} is independent of $\epsilon$, which means that the family $\{Du^{\epsilon,k,m}\}_{0<\epsilon<1}$ is uniformly bounded in $W^{1,2}(\Om_m;\Rn)$. By weak compactness, we can select a subsequence $\{Du^{\epsilon_j,k,m}\}_{j=1}^\infty$ 
such that we have the weak convergence
\begin{align} \label{eq:Weak-convergence-of-derivatives-in-W12-1}
    Du^{\epsilon_j,k,m}\xrightharpoonup{j\to\infty} U 
    \quad\text{in } W^{1,2}(\Om_m;\Rn) 
\end{align}
for some $U\in W^{1,2}(\Om_m;\Rn)$. We claim that $U=Du^{k,m}$, where $u^{k,m}$ solves \eqref{eq:DirichletProblemHomogeneous-regularized-2}.
This follows easily from the convergences \eqref{eq:Convergence-of-u-1} and \eqref{eq:Weak-convergence-of-derivatives-in-W12-1} and the uniqueness of weak limit.

Indeed, if $2\leq p<3+\frac{2}{n-2}$, then \eqref{eq:Convergence-of-u-1} implies that
$$ Du^{\epsilon_j,k,m}\xrightarrow{j\to\infty} Du^{k,m}
\quad\text{in }L^2(\Om_m;\Rn) $$
and in particular,
\begin{equation} \label{eq:Weak-convergence-of-derivatives-in-L2-1}
    Du^{\epsilon_j,k,m}\xrightharpoonup{j\to\infty} Du^{k,m}
\quad\text{in }L^2(\Om_m;\Rn).
\end{equation}
Since weak limit is unique, \eqref{eq:Weak-convergence-of-derivatives-in-W12-1} and \eqref{eq:Weak-convergence-of-derivatives-in-L2-1}
imply that $U=Du^{k,m}$. On the other hand, if $1<p<2$, then \eqref{eq:Weak-convergence-of-derivatives-in-W12-1} implies that
\begin{align} \label{eq:Weak-convergece-of-derivatives-in-Lp-1a}
    Du^{\epsilon_j,k,m}\xrightharpoonup{j\to\infty} U 
    \quad\text{in } L^p(\Om_m;\Rn)
\end{align}
Trivially \eqref{eq:Convergence-of-u-1} implies that
\begin{align} \label{eq:Weak-convergece-of-derivatives-in-Lp-1b}
    Du^{\epsilon_j,k,m}\xrightharpoonup{j\to\infty} Du^{k,m}
    \quad\text{in } L^p(\Om_m;\Rn).
\end{align}
Since weak limit is unique, \eqref{eq:Weak-convergece-of-derivatives-in-Lp-1a} and \eqref{eq:Weak-convergece-of-derivatives-in-Lp-1b} imply that $U=Du^{k,m}$.

We conclude that $Du^{k,m}\in W^{1,2}(\Om_m;\Rn)$ and $Du^{\epsilon_j,k,m}\xrightharpoonup{j\to\infty} Du^{k,m}$ in $W^{1,2}(\Om_m;\Rn)$. We let $j\to\infty$ in \eqref{eq:Global-estimate-with-uekm} to conclude that
\begin{align*}
    \|Du^{k,m}\|_{W^{1,2}(\Om_m;\Rn)}
    &\leq 
    \liminf_{j\to\infty}\,
    \|Du^{\epsilon_j,k,m}\|_{W^{1,2}(\Om_m;\Rn)} \\
    &\leq
    \liminf_{j\to\infty}\,
    C\Big(\|D\varphi_{k}\|_{W^{1,2}(\Om_m;\Rn)}+\|D\varphi_{k}\|_{L^p(\Om_m;\Rn)}+\epsilon
    \Big) \\
    &=
    C\Big(\|D\varphi_k\|_{W^{1,2}(\Om_m;\Rn)}+\|D\varphi_k\|_{L^p(\Om_m;\Rn)}
    \Big).
\end{align*}
So the final conclusion of Step 1 is the estimate
\begin{align} \label{eq:Global-estimate-with-ukm}
    \|Du^{k,m}\|_{W^{1,2}(\Om_m;\Rn)}
    \leq 
    C\Big(\|D\varphi_k\|_{W^{1,2}(\Om_m;\Rn)}+\|D\varphi_k\|_{L^p(\Om_m;\Rn)}
    \Big).
\end{align}

\underline{Step 2}: Let $k\to\infty$. For $m=1,2,\ldots$ fixed, consider the sequence $\{Du^{k,m}\}_{k=1}^\infty$. By the estimate \eqref{eq:Global-estimate-with-ukm} and the convergence \eqref{eq:Convergece-of-boundary-data-in-Om},
\begin{equation} \label{eq:Global-estimate-with-ukm-1}
\begin{aligned}
    \|Du^{k,m}\|_{W^{1,2}(\Om_m;\Rn)}
    \leq C\Big(\|D(E\varphi)\|_{W^{1,2}(\Om_m;\Rn)}+\|D(E\varphi)\|_{L^p(\Om_m;\Rn)}
    +1\Big).
\end{aligned}
\end{equation}
The right hand side of \eqref{eq:Global-estimate-with-ukm-1} is independent of $k$, which means that the sequence $\{Du^{k,m}\}_{k=1}^\infty$ is uniformly bounded in $W^{1,2}(\Om_m;\Rn)$. By weak compactness, we can select a subsequence $\{Du^{k_j,m}\}_{j=1}^\infty$ 
such that we have the weak convergence
\begin{align} \label{eq:Weak-convergence-of-derivatives-in-W12-2}
    Du^{k_j,m}\xrightharpoonup{j\to\infty} U 
    \quad\text{in } W^{1,2}(\Om_m;\Rn) 
\end{align}
for some $U\in W^{1,2}(\Om_m;\Rn)$. By a similar argument as in Step 1, the convergences \eqref{eq:Convergence-of-u-2} and \eqref{eq:Weak-convergence-of-derivatives-in-W12-2}, together with the uniqueness of weak limit, imply that $U=Du^{m}$, where $u^{m}$ solves \eqref{eq:DirichletProblemHomogeneous-regularized-3}.

We conclude that $Du^{m}\in W^{1,2}(\Om_m;\Rn)$ and $Du^{k_j,m}\xrightharpoonup{j\to\infty} Du^{m}$ in $W^{1,2}(\Om_m;\Rn)$. We let $j\to\infty$ in \eqref{eq:Global-estimate-with-ukm} to conclude that
\begin{align*}
    \|Du^{m}\|_{W^{1,2}(\Om_m;\Rn)} 
    &\leq 
    \liminf_{j\to\infty}\,
    \|Du^{k_j,m}\|_{W^{1,2}(\Om_m;\Rn)} \\
    &\leq
    \liminf_{j\to\infty}\,
    C\Big(\|D\varphi_{k_j}\|_{W^{1,2}(\Om_m;\Rn)}+\|D\varphi_{k_j}\|_{L^p(\Om_m;\Rn)}\Big) \\
    &=
    C\Big(\|D(E\varphi)\|_{W^{1,2}(\Om_m;\Rn)}+\|D(E\varphi)\|_{L^p(\Om_m;\Rn)}
    \Big).
\end{align*}
Here we also employed the convergence \eqref{eq:Convergece-of-boundary-data-in-Om}.
The final conclusion of Step 2 is the estimate
\begin{align} \label{eq:Global-estimate-with-um}
    \|Du^{m}\|_{W^{1,2}(\Om;\Rn)}
    \leq 
    C\Big(\|D(E\varphi)\|_{W^{1,2}(\Om_m;\Rn)}+\|D(E\varphi)\|_{L^p(\Om_m;\Rn)}
    \Big).
\end{align}

\underline{Step 3}: Let $m\to\infty$. Consider the sequence $\{Du^{m}\}_{m=1}^\infty$. By the estimate \eqref{eq:Global-estimate-with-um}
\begin{equation} \label{eq:Global-estimate-with-um-1}
\begin{aligned}
    \|Du^{m}\|_{W^{1,2}(\Om;\Rn)}
    \leq C\Big(\|D(E\varphi)\|_{W^{1,2}(\Rn;\Rn)}+\|D(E\varphi)\|_{L^{p}(\Rn;\Rn)}\Big).
\end{aligned}
\end{equation}

The right hand side of \eqref{eq:Global-estimate-with-um-1} is independent of $m$, which means that the sequence $\{Du^{m}\}_{m=1}^\infty$ is uniformly bounded in $W^{1,2}(\Om;\Rn)$. By weak compactness, we can select a subsequence $\{Du^{m_j}\}_{j=1}^\infty$
such that we have the weak convergence
\begin{align} \label{eq:Weak-convergence-of-derivatives-in-W12-3}
    Du^{m_j}\xrightharpoonup{j\to\infty} U 
    \quad\text{in } W^{1,2}(\Om;\Rn). 
\end{align}
By a similar argument as in Step 1 and Step 2, the convergences \eqref{eq:Convergence-of-u-3} and \eqref{eq:Weak-convergence-of-derivatives-in-W12-3}, together with the uniqueness of weak limit, imply that $U=Du$, where $u$ solves \eqref{eq:Dirichlet-problem-homogeneous}.

We conclude that $Du\in W^{1,2}(\Om;\Rn)$ and $Du^{m_j}\xrightharpoonup{j\to\infty} Du$ in $W^{1,2}(\Om;\Rn)$. We let $j\to\infty$ in \eqref{eq:Global-estimate-with-um} to conclude that
\begin{align*}
    \|Du\|_{W^{1,2}(\Om;\Rn)}
    &\leq 
    \liminf_{j\to\infty}\,
    \|Du^{m_j}\|_{W^{1,2}(\Om;\Rn)} \\
    &\leq
    \liminf_{j\to\infty}\,
    C\Big(\|D(E\varphi)\|_{W^{1,2}(\Om_m;\Rn)}+\|D(E\varphi)\|_{L^p(\Om_m;\Rn)}\Big) \\
    &=   
    C\Big(\|D\varphi\|_{W^{1,2}(\Om;\Rn)}+\|D\varphi\|_{L^p(\Om;\Rn)}\Big) 
\end{align*}
which is the desired estimate. The proof is finished.
\end{proof}

\bibliographystyle{amsplain}
\bibliography{bibliography}
\end{document}